\documentclass[reqno]{amsart}

\usepackage{amsmath}
\usepackage{amssymb}
\usepackage[margin=1in]{geometry}
\usepackage{xcolor}
\usepackage{comment}
\usepackage{graphicx}
\usepackage{empheq}

\theoremstyle{plain}
\newtheorem{thm}{Theorem}[section]
\newtheorem{lem}{Lemma}[section]
\newtheorem{prop}{Proposition}[section]
\newtheorem{cor}{Corollary}[section]

\theoremstyle{definition}
\newtheorem{defn}{Definition}[section]

\theoremstyle{remark}

\newtheorem{rmk}{Remark}[section]

\newcommand{\norm}[1]{\left\lVert#1\right\rVert}

\newcommand{\tr}[1]{\operatorname{tr}\left(#1\right)}
\newcommand{\diag}[1]{\left[#1\right]}

\newcommand{\R}{\mathbb{R}}
\usepackage{cite}
\usepackage{hyperref}
\hypersetup{ colorlinks = true, urlcolor = blue, linkcolor = blue, citecolor = red }
\usepackage{enumitem}
\usepackage{nameref}
\makeatletter
\let\orgdescriptionlabel\descriptionlabel
\renewcommand*{\descriptionlabel}[1]{%
  \let\orglabel\label
  \let\label\@gobble
  \phantomsection
  \edef\@currentlabel{#1\unskip}%
  \let\label\orglabel
  \orgdescriptionlabel{#1}%
}
\numberwithin{equation}{section}

\begin{document}
	
\title[ ABP estimate and Harnack inequality for fully nonlinear pseudo-$p$-Laplacian]{  ABP estimate and Harnack inequality for a class of degenerate fully nonlinear pseudo-$p$-Laplacian equations}

\author[Byun]{Sun-Sig Byun}
\address{Department of Mathematical Sciences and Research Institute of Mathematics,
	Seoul National University, Seoul 08826, Republic of Korea}
\email{byun@snu.ac.kr}

\author[Kim]{Hongsoo Kim}
\address{Department of Mathematical Sciences, Seoul National University, Seoul 08826, Republic of Korea}
\email{rlaghdtn98@snu.ac.kr}

\thanks {S.-S. Byun was supported by Mid-Career Bridging
Program through Seoul National University. H. Kim was supported by the National Research Foundation of Korea(NRF) grant funded by the Korea government [Grant No. 2022R1A2C1009312].}

\makeatletter
\@namedef{subjclassname@2020}{\textup{2020} Mathematics Subject Classification}
\makeatother

\everymath{\displaystyle}

\begin{abstract}
	We prove Aleksandrov-Bakelman-Pucci estimates and Harnack inequalities for viscosity solutions of a class of degenerate  fully nonlinear pseudo-$p$-Laplacian equations in nondivergence form.
    Our main approach is an adaptation of the sliding paraboloid method with anisotropic functions tailored to the coordinatewise degeneracy.
\end{abstract}

\maketitle

\section{Introduction} \label{sec1}
In this paper, we study Aleksandrov-Bakelman-Pucci estimates and Harnack inequalities for viscosity solutions of the following degenerate pseudo-$p$-Laplacian inequalities in nondivergence form:
\begin{align} \label{PDE}
    \begin{cases}
    \mathcal{M}^-_{\lambda,\Lambda} \left( \diag{|D_iu|^{p/2}} D^2u \diag{|D_iu|^{p/2}} \right) -\Lambda|Du|^{p+1} \leq f(x), \\
     \mathcal{M}^+_{\lambda,\Lambda} \left( \diag{|D_iu|^{p/2}} D^2u \diag{|D_iu|^{p/2}} \right)  +\Lambda|Du|^{p+1} \geq f(x),
    \end{cases}
     \quad \text{in } B_1,
\end{align}
where $p \geq 0$, $0<\lambda \leq \Lambda$, $f \in L^n(B_1)$ and $\diag{|z_i|^{p/2}} =\operatorname{diag}(|z_i|^{p/2})$ is the diagonal matrix with entries $|z_i|^{p/2}$ on the diagonal.

The following are examples that satisfy the aforementioned inequalities.
\begin{enumerate}
    \item the degenerate pseudo-$(p+2)$-Laplacian equation. 
        \begin{align*}
        \tilde{\Delta}_{p+2}u(x)=  \frac{1}{p}\sum_i \partial_i(|\partial_iu|^{p+1}\partial_{i}u)=\sum_i |\partial_iu|^{p}\partial_{ii}u = f(x).    \end{align*}
    \item For some symmetric matrix function $(a_{ij}(x))$ with eigenvalues in $[\lambda,\Lambda]$ and some $|b_i(x)|\leq\Lambda$, 
        \begin{align*}
        \sum_{i,j} a_{ij}(x)|\partial_iu|^{p/2}|\partial_ju|^{p/2} \partial_{ij}u + \sum_i b_i(x)\partial_iu= f(x).    \end{align*}
\end{enumerate}
Operators of pseudo-$p$-Laplacian type arise naturally in anisotropic diffusion and orthotropic models, where the response of the medium depends on individual coordinate directions rather than on the Euclidean gradient.
Such structures appear in variational problems with directional growth, transport phenomena in layered or composite media, and anisotropic regularization mechanisms in applied mathematics.
The coordinatewise degeneracy encoded in \eqref{PDE} reflects situations in which diffusion may vanish or weaken along specific directions while remaining effective in others.

Therefore, the pseudo-$p$-Laplacian equation has long been a central theme in the analysis of partial differential equations and has been the subject of extensive research.
In the divergence form, the pseudo-$p$-Laplacian arises as the Euler–Lagrange equation corresponding to the anisotropic functional:
\begin{align*}
    w\rightarrow\int\sum_i|D_iw|^{p}.
\end{align*}
The Lipschitz regularity of weak solutions of pseudo-$p$-Laplacian was established by Bousquet, Brasco, Leone, and Verde \cite{Bousquet18} via Moser’s iteration.
For further results of higher regularity of pseudo-$p$-Laplacian, see \cite{Bousquet20,Bousquet16,Lindquist18,Brasco17,Bousquet23}.
In contrast, the basic regularity results such as H\"older continuity and Harnack inequalities follow from standard methods for the $p$-laplace equation (see Chapter 6 and 7 of \cite{Giusti03}) due to the equivalence of the two integrands
\begin{align*}
    \sum_i|D_iu|^{p} \approx |Du|^p.
\end{align*}
In the nondivergence setting, Lipschitz regularity for viscosity solutions of pseudo-$p$-Laplacian was established by Demengel \cite{Demengel162} based on the Ishii-Lions method, and Birindelli and Demengel \cite{Demengel16} proved existence and Lipschitz regularity for a broader class of equations including the pseudo-$p$-Pucci's operator, which is the motivation for the inequalities \eqref{PDE}. 
However, the basic regularity theory in nondivergence case does not appear to be standard and differs from that of the divergence case.
The reason is that we do not have the equivalence for the two nondivergence forms:
\begin{align*}
    \tilde{\Delta}_{p+2}u=\sum_i|D_iu|^{p} D_{ii}u\not\approx |Du|^p\sum_iD_{ii}u \approx\Delta_{p+2}u.
\end{align*}
For the $p$-Laplacian, the equation is uniformly elliptic where $|Du|>c$ for some $c>0$.
Therefore, one can apply the results of Imbert and Silvestre \cite{Imbert16} or Mooney \cite{Mooney15}, who established Harnack inequalities for equations which hold only where the gradient is large.

On the other hand, the pseudo-$p$-Laplacian exhibits coordinatewise degeneracy: it loses ellipticity whenever a partial derivative vanishes, that is, at points where $D_ju=0$ for some $j \in \{1,\cdots,n\}$.
This degeneracy is unbounded and occurs on a set much larger than $\{|Du|=0\}$, so the preceding results do not apply directly.
More precisely, if $D_j u(x)=0$ for some $j$, the second derivative $D_{jj}u(x)$ may become unbounded; in contrast, if there exists a direction $i$ with $D_i u(x)\neq 0$, then $D_{ii}u(x)$ remains controlled and no loss of regularity occurs along the $i$-th direction. 
Consequently, discarding data at points where $D_j u(x)=0$ also neglects the information carried by nondegenerate directions $i$ with $D_i u(x)\neq 0$,  and this prevents us from establishing regularity.
Therefore, it is essential to exploit the information from nondegenerate directions at points of degeneracy.

Despite these difficulties caused by strong anisotropy and coordinatewise degeneracy, we develop the basic regularity of a class of degenerate pseudo-$p$-Laplacian in nondivergence form.
The first main result is the Aleksandrov-Bakelman-Pucci estimates for the psuedo-$p$-laplacian.

\begin{thm} \label{Main1}
Let $p \ge 0$. 
Let $u \in C(\Omega)$ satisfy 
\begin{align} \label{supsol}&\mathcal{M}^+_{\lambda,\Lambda} \left( \diag{|D_iu|^{p/2}} D^2u \diag{|D_iu|^{p/2}} \right) +\Lambda|Du|^{p+1} \geq f(x) \quad \text{ in } \Omega.
\end{align}
Then there exists a constant $C = C(n,\lambda,\Lambda,p,d) > 0$ such that
\[
\sup_\Omega u
\le
\sup_{\partial\Omega} u^+
+ C
\|f^{-}\|_{L^n(\Gamma^+(u))}^{\frac{1}{1+p}},
\]
where $d=\operatorname{diam}(\Omega)$ and $\Gamma^+(u)$ is the upper contact set defined in \eqref{upp}.

Analogously, if $u \in C(\Omega)$ satisfies
\begin{align*}
\mathcal{M}^-_{\lambda,\Lambda} \left( \diag{|D_iu|^{p/2}} D^2u \diag{|D_iu|^{p/2}} \right) -\Lambda|Du|^{p+1} \leq f(x) \quad \text{ in } \Omega.
\end{align*}
Then there exists a constant $C=C(n,\lambda,\Lambda,p,d)>0$ such that
\begin{align*}
\sup_{\Omega}u^- \leq \sup_{\partial \Omega}u^- +C \norm{f^+}_{L^n(  \Gamma^+(-u))}^{\frac{1}{1+p}}.
\end{align*}
\end{thm}

The second main result is an Harnack inequality with $L^n$ data in the right-hand side.

\begin{thm} \label{Main2}
Let $p \geq 0$. Assume that $u \in C(B_1)$ is nonnegative and satisfies \eqref{PDE} in the viscosity sense in $B_1$, where
$f \in C(B_1)\cap L^n(B_1)$.
Then there exists a constant $C = C(n,\lambda,\Lambda,p) > 1$ such that
$$
\sup_{B_{1/2}} u
\;\le\;
C\Bigl(
\inf_{B_{1/2}} u
\;+\;
\|f\|_{L^n(B_1)}^{\frac{1}{1+p}}
\Bigr).
$$
\end{thm}

The direct consequence of the Harnack inequality is the H\"older regularity.

\begin{cor} \label{maincor}
Let $u \in C(B_1)$ satisfy \eqref{PDE} in the viscosity sense and $f \in C(B_1)\cap L^n(B_1)$.
Then there exist constants $\alpha = \alpha(n,\lambda,\Lambda,p) > 0$ and
$C = C(n,\lambda,\Lambda,p) > 1$ such that
\[
\|u\|_{C^\alpha(B_{1/2})}
\le
C\Big(
\|u\|_{L^\infty(B_1)}
+ \|f\|_{L^n(B_1)}^{\frac{1}{1+p}}
\Big).
\]
\end{cor}

We now outline the main ideas of our proof.
The main idea of the ABP estimate is using the area formula to estimate the image of $(|D_iu|^{p}D_iu)$, instead of $Du$, which is motivated in \cite{Argiolas11}.
See also \cite{Baasandorj242,Imbert11,Davila09,Kim252} for more results about the ABP estimates.

Our primary tool for the Harnack inequality is the sliding paraboloid method, originally introduced by Cabre \cite{Cabre97} and developed by Savin \cite{Savin07}.
We slide the paraboloid from below or above until it touches the graph of a solution $u$ under consideration.
Then the touching points have some good properties and by applying the area formula, the measure of touching points can be estimated in terms of the measure of the corresponding vertex points.
This approach has been successfully employed to prove the Krylov-Safanov results and $W^{2,\delta}$ regularity results for elliptic and parabolic equations.
See also \cite{Li18,Mooney15,Baasandorj24,Zhenyu25,Wang13,Colombo14,Chang-Lara23,Le18}.
For degenerate equations, by choosing a suitable functions such as cusps $|x|^{1/2}$ in \cite{Imbert16,Banerjee22,Byun252}, or other power functions $|x|^{1+\alpha}$ in \cite{Byun25,Vedansh25} instead of the standard paraboloid $|x|^2$, it is possible to handle the degeneracy of the equations.
In our setting, we take the following anisotropic function as our `paraboloid':
\begin{align*}
        \varphi(x) = -\frac{1+p}{2+p}\left(|x_1|^{1+\frac{1}{1+p}} + \cdots +|x_n|^{1+\frac{1}{1+p}} \right),
\end{align*}
to overcome the degeneracy of the equations.
Observe that $\varphi(x)$ is a standard paraboloid when $p=0$.
The motivation of the choice of $\varphi$ is that it satisfies the following identity: $-(1+p)\diag{|D_i\varphi|^{p/2}}D^2\varphi\diag{|D_i\varphi|^{p/2}} = I_n$.
It should be noted, however, that $\varphi(x)$ fails to be $C^2$ on the set $\bigcup\{z_i=0\}$, which causes additional difficulties in the proof.

To explain this in more detail, we slide a paraboloid $\varphi(z-y)$ with vertex $y$ from below until it touches $u$ at $x$.
Since $Du(x)=D\varphi(x-y)$, we have $$y_i = x_i+|D_iu|^pD_iu(x).$$
for each $i=\{1,\cdots, n\}$.
Therefore, by differentiating with respect to $x$ and taking determinants, we obtain
\begin{align*}
    \det (D_xy) = \det(I+ (p+1)\diag{|D_iu|^{p/2}}D^2u\diag{|D_iu|^{p/2}}).
\end{align*}
From the fact $D^2u(x)\geq D^2\varphi(x-y)$ and using the equalities \eqref{PDE}, we deduce that $\det (D_xy) \leq C(1+|f(x)|^n)$, and then apply area formula to get the desired result.
Note that this argument can be applied where $D^2\varphi(x-y)$ is well-defined. 

If the inequalities degenerate at a point $x$, for example only on $n$-th direction, say $D_nu(x)=0$ and $D_iu(x)\neq0$ for $i\neq n$, then we have $x_n-y_n=0$ which causes $D_{nn}\varphi(x-y)$ to blow up.
In this case, we discard $n$-th direction and carry out the argument in the
remaining nondegenerate directions by restricting the argument to the $(n-1)$-dimensional slice $H=\{z\in\mathbb{R}^n:z_n=y_n\}$.
Observe that the restricted paraboloid $\varphi_{\{n\}}(x-y)$ as a function on the $(n-1)$-dimensional space $H$ is smooth at $x$ since $x_i-y_i\neq0$ for $i\neq n$.
Also, putting $D_nu(x)=0$, the inequalities \eqref{PDE} in $n$ dimensions reduce to the same form of the inequalities in $(n-1)$ dimensions.
Hence, by repeating the above argument with the dimension reduced from $n$ to $(n-1)$ and then applying the Fubini's theorem to sum over the slices, we obtain the desired result, Lemma \ref{measlem}. 

We next construct a barrier function, which is needed to establish the doubling property, Lemma \ref{doublem}.
By taking the standard barrier function $\Phi(x)=|x|^{-a}$ for sufficiently large $a>0$, one can verify that it is a subsolution in the sense that $\mathcal{M}^-\left(\diag{|D_i\Phi|^{p/2}}D^2\Phi\diag{|D_i\Phi|^{p/2}}\right) -\Lambda|D\Phi|^{p+1} >1$.
If we consider $L^\infty$ data of $f$ with $\norm{f}_{L^\infty} <1$, then the doubling property follows directly from the comparison principle.
However, this approach fails when only $L^n$ data for $f$ are available.
Instead, we slide the barrier and use the area formula in order to measure the set of touching points by the set of vertex points as in Lemma \ref{measlem}.
Our selection of the barrier function is the following anisotropic function
\begin{align*}
    \Phi(x)=|\varphi(x)|^{-a} =\left(|x_1|^{1+\frac{1}{1+p}} + \cdots +|x_n|^{1+\frac{1}{1+p}} \right)^{-a},
\end{align*}
for sufficiently large $a>0$.
The reason for using $|\varphi(x)|$ instead of $|x|$ is that it matches with the degeneracy of the equation.
We also note that $\Phi(x)$ fails to be $C^2$ on the set $\bigcup\{z_i=0\}$, so that a similar restriction argument at degenerate points, as in Lemma \ref{measlem}, is required.
Once Lemma \ref{measlem} and Lemma \ref{doublem} is proved, the proof of Harnack inequalities follows by standard arguments (see \cite{Caffarelli95} or \cite{Imbert16}). 

The paper is organized as follows.
In Section \ref{sec2}, we introduce the notation and preliminaries.
In Section \ref{sec3}, we prove the Aleksandrov-Bakelman-Pucci estimates, Theorem \ref{Main1}.
In Section \ref{sec4}, we prove the measure estimate, Lemma \ref{measlem}.
In Section \ref{sec5}, we prove the doubling property, Lemma \ref{doublem}, and complete our proof of the main theorems.

\section{Notations and Preliminaries} \label{sec2}

Throughout this paper, for $x \in \mathbb{R}^n$, we write $x=(x_1,\cdots ,x_n)$. We denote by $B_r(x_0) = \{ x \in \mathbb{R}^n : |x-x_0| <r \}$ for a point $x_0 \in \mathbb{R}^n$ and $B_r = B_r(0)$.
We write a cube as $Q_r^n(x_0) =Q_r(x_0) = \{ x \in \mathbb{R}^n : |x_i-(x_0)_i| <r \text{ for any } i\in\{ 1, \cdots,n\}\}$ and $Q_r = Q_r(0)$.

For $a \in \mathbb{R}$ and $i\in\{ 1, \cdots,n\}$, we write a hyperplane as $\{z_i = a\} := \{z\in\mathbb{R}^n : z_i=a\}$.
$S(n)$ denotes the space of symmetric $n \times n$ real matrices and $I_n$ denotes the $n \times n$ identity matrix.
For $a,b >0$, we write $a\approx b$ if there exists a constant $C>1$ such that $\frac{1}{C}b \leq a \leq Cb$.
For a function $f:\Omega \rightarrow \R$, we say $f^+ = \max\{f,0\}$ and $f^- = \max\{-f,0\}$.

For $a=(a_i) \in \mathbb{R}^n$, we always write the diagonal matrix with entries $a_i$ as
\begin{align*}
    \diag{a_i} := \operatorname{diag}(a_1,\cdots,a_n) \in S(n).
\end{align*}

We recall the definition of the inequalities \eqref{PDE} in the viscosity sense from \cite{Ishii92, Caffarelli95} as follows.
\begin{defn}
Let $f \in C(B_1)$.
We say that $u \in C(\overline{B}_1)$ satisfies 
$$\mathcal{M}^\pm_{\lambda,\Lambda} \left( \diag{|D_iu|^{p/2}} D^2u \diag{|D_iu|^{p/2}} \right) \pm\Lambda|Du|^{p+1} \leq f(x) \quad\text{in}\ B_{1} \quad (\text{resp.} \geq) $$
in the viscosity sense, if for any $x_0 \in B_1$ and any test function $\psi \in C^2(B_1)$ such that $u-\psi$ has a local minimum (resp. maximum) at $x_0$, then
$$\mathcal{M}^\pm_{\lambda,\Lambda} \left( \diag{|D_i\psi(x_0)|^{p/2}} D^2\psi(x_0) \diag{|D_i\psi(x_0)|^{p/2}} \right) \pm\Lambda|D\psi(x_0)|^{p+1} \leq f(x_0)  \quad (\text{resp.} \geq).$$ 
\end{defn}

We recall the definition and some properties of the Pucci operator (see \cite{Caffarelli95}).
\begin{defn}
    For given $0<\lambda \leq \Lambda$, we define the Pucci operators $\mathcal{M}_{\lambda,\Lambda}^{\pm} : S(n) \rightarrow \mathbb{R}$ as follows:
\begin{align*}
    \mathcal{M}_{\lambda,\Lambda}^{+}(M) := \lambda \sum_{e_i(M)<0}e_i(M) + \Lambda \sum_{e_i(M)>0}e_i(M), \\
    \mathcal{M}_{\lambda,\Lambda}^{-}(M) := \Lambda \sum_{e_i(M)<0}e_i(M) + \lambda \sum_{e_i(M)>0}e_i(M),
\end{align*}
where $e_i(M)$'s are the eigenvalues of $M$.
We also abbreviate $\mathcal{M}^{\pm}_{\lambda,\Lambda}$ as $\mathcal{M}^{\pm}$.
\end{defn}

\begin{prop} 
For any $M,N \in S(n)$, we have
    \begin{enumerate}
        \item $ \mathcal{M}_{\lambda,\Lambda}^{-}(M) +\mathcal{M}_{\lambda,\Lambda}^{-}(N)\leq \mathcal{M}_{\lambda,\Lambda}^{-}(M+N) \leq \mathcal{M}_{\lambda,\Lambda}^{-}(M) +\mathcal{M}_{\lambda,\Lambda}^{+}(N)$. \\
        \item $ \mathcal{M}_{\lambda,\Lambda}^{+}(M) +\mathcal{M}_{\lambda,\Lambda}^{-}(N)\leq \mathcal{M}_{\lambda,\Lambda}^{+}(M+N) \leq \mathcal{M}_{\lambda,\Lambda}^{+}(M) +\mathcal{M}_{\lambda,\Lambda}^{+}(N).$ 
    \end{enumerate}
\end{prop}

\begin{rmk} \label{scaling}
    (Scaling) Observe that if $u$ satisfies $\mathcal{M}^-_{\lambda,\Lambda} \left( \diag{|D_iu|^{p/2}} D^2u \diag{|D_iu|^{p/2}} \right) -\Lambda|Du|^{p+1} \leq f(x)$ in $B_1$, then the function $\tilde{u}(x)=\frac{u(rx)}{K}$
    satisfies the following inequalities:
    \begin{align*}
        \mathcal{M}^-_{\lambda,\Lambda} \left( \diag{|D_iv|^{p/2}} D^2v \diag{|D_iv|^{p/2}} \right) -r\Lambda|Du|^{p+1}  \leq \tilde{f}(x):=\frac{r^{p+2}}{M^{p+1}}f(x).
    \end{align*}
\end{rmk}

Finally, we write the Calder\'on-Zygmund cube decomposition (see  \cite{Caffarelli95}, Lemma 4.2).
We split the cube $Q_1$ into $2^n$ subcubes of half the side length and repeat this splitting on each new cube.
For any dyadic cube $Q$, we say the $\tilde{Q}$ is the predecessor of $Q$ if $Q$ is one of the $2^n$ cubes split from $\tilde{Q}$.
\begin{lem}[\cite{Caffarelli95}] \label{CZ}
    Let $E \subset F\subset Q_1$ be measurable sets and $0<\delta<1$ such that
    \begin{enumerate}
        \item $|E| \leq \delta$,
        \item If $Q$ is a dyadic cube such that $|E\cap Q| > \delta |Q|$, then $\tilde{Q}\subset F$.
    \end{enumerate}
    Then $|E| \leq \delta|F|$.
\end{lem}

\section{Proof of ABP estimates} \label{sec3}
In this section, we prove the Aleksandrov-Bakelman-Pucci estimates, Theorem \ref{Main1}.
We first introduce the notion of upper contact set of $u$.
For $u \in C(\Omega)$ and $r>0$ as
\begin{align} \label{upp}
\begin{split}
    \Gamma^+(u) = \{ y \in \Omega : \exists p \in \R^n \text{ such that } u(x) \leq u(y) + \langle{p,x-y} \rangle, \  \forall x \in \Omega \}, \\ 
    \Gamma^+_r(u) = \{ y \in \Omega : \exists p \in \overline{B_r} \text{ such that } u(x) \leq u(y) + \langle{p,x-y} \rangle, \  \forall x \in \Omega \}.
\end{split}
\end{align}

\begin{proof}[Proof of Theorem \ref{Main1}]
We only prove the first inequality since the other can be proved similarly.
We assume that $u \in C^2(\Omega) \cap C(\overline{\Omega})$ by using approximating based on the regularization process via the sup convolutions. (See \cite{Caffarelli96}.)
Let us define
\begin{align*}
    r_0 = \frac{\sup_{\Omega}u - \sup_{\partial \Omega}u^+}{d}.
\end{align*}
Then using the same method in \cite{Argiolas11}, for $r<r_0$, we obtain
\begin{align*}
    \Gamma^+_r(u) \subset\subset \Omega, \quad \text{ and } \quad  Q_{r/\sqrt{n}}\subset B_r \subset Du(\Gamma^+_r(u)).
\end{align*}
We also consider the map $G : \R^n \rightarrow \R^n$ as
\begin{align*}
    G(z)=(|z_i|^{p}z_i).
\end{align*}
By direct calculation, we get $DG(z) = (p+1)\diag{|z_i|^{p}}$.
Moreover, we have $G(Q_{r/\sqrt{n}})=Q_{(r/\sqrt{n})^{p+1}} \subset G(Du)(\Gamma^+_r(u))$.
Therefore, applying the area formula to Lipschitz map $G(Du)$, we obtain
\begin{align} \label{areafor}
    \int_{Q_{(r/\sqrt{n})^{p+1}}} h(p) \, dp &\leq \int_{\Gamma^+_r(u)} |\det (DG(Du) \cdot D^2u)| h(G(Du)) \, dx.
\end{align}
for some function $h : \R^n \rightarrow [0,\infty)$.
Since $D^2u \leq 0$ on $\Gamma^+(u)$, we get $\diag{|D_iu|^{p/2}}D^2u\diag{|D_iu|^{p/2}} \leq0$ on $\Gamma^+(u)$.
Thus, we obtain
\begin{align*}
    \mathcal{M}^+_{\lambda,\Lambda}  \left( \diag{|D_iu|^{p/2}} D^2u \diag{|D_iu|^{p/2}} \right)  =\lambda\tr{\diag{|D_iu|^{p/2}} D^2u \diag{|D_iu|^{p/2}}} \quad \text{ on } \Gamma^+(u).
\end{align*}
Using that $u$ satisfies \eqref{supsol}, we have
\begin{align*}
    \tr{-\diag{|D_iu|^{p/2}} D^2u \diag{|D_iu|^{p/2}}} \leq \frac{1}{\lambda}\left(f^- + \Lambda|Du|^{p+1} \right)\quad \text{ on } \Gamma^+(u).
\end{align*}
Recalling that $\det(A) \leq \left(\frac{\tr{A}}{n}\right)^n$ for positive definite matrix $A$, we get
\begin{align*}
    |\det (DG(Du) \cdot D^2u)| &=(p+1)^n|\det(\diag{|D_iu|^{p}}D^2u)|\\
    &=(p+1)^n|\det(\diag{|D_iu|^{p/2}}D^2u\diag{|D_iu|^{p/2}})| \\
    &\leq\left|(p+1)\frac{\tr{-\diag{|D_iu|^{p/2}}D^2u\diag{|D_iu|^{p/2}}}}{n}\right|^n \\
    &\leq \left|\frac{p+1}{n\lambda}(f^- + \Lambda|Du|^{p+1}) \right|^n.
\end{align*}
Thus, choosing $h(z) = \frac{1}{\kappa^n+|z|^n}$ for some $\kappa>0$, we obtain

\begin{align} \label{ABP1}
    \int_{\Gamma^+_r(u)} |\det (DG(Du) \cdot D^2u)| h(G(Du)) \, dx &\leq C \int_{\Gamma^+_r(u)}\frac{(f^-)^n + |Du|^{n(p+1)}}{\kappa+|(|D_iu|^pD_iu)|^n} \, dx  \nonumber\\
    &\leq C \int_{\Gamma^+_r(u)}\frac{(f^-)^n}{\kappa^n}+ C \, dx.
\end{align}
for some $C=C(n,p)>0$.
Here we used that $|z|^{p+1} \approx \sum|z_i|^{p+1}$ for any $z \in \R^n$ depending on $n$ and $p$.
Moreover, by direct calculation, we have
\begin{align} \label{ABP2}
    \int_{Q_{(r/\sqrt{n})^{p+1}}} h(z) \, dz \geq \int_{B_{(r/\sqrt{n})^{p+1}}} \frac{1}{\kappa^n+|z|^n} \, dz =n \omega_n \ln \left( 1+\frac{r^{n(p+1)}}{\kappa^n n^{n(p+1)/2}}\right)
\end{align}
We select $\kappa =\norm{f^-}_{L^n(\Gamma^+(u))} +\epsilon$ and combine \eqref{ABP1} and \eqref{ABP2} to \eqref{areafor} with $\epsilon \rightarrow 0$.
Then, we obtain
$r \leq C\norm{f^-}_{L^n(\Gamma^+(u))}^{\frac{1}{1+p}}$ for some $C=C(n,p,\lambda,\Lambda,d)>0$, which concludes the proof.
\end{proof}

\section{Basic measure estimate} \label{sec4}
We focus on the proof of Harnack inequality, Theorem \ref{Main2}.
In this section, we prove the following measure estimate lemma.

\begin{lem} \label{measlem}
    There exist $\delta<1$, $\mu<1$ and $M_1>1$ depending on $n,\lambda, \Lambda$ and $p \geq 0$ such that if $u \in C(Q_1)$ satisfies
    \begin{align*}
    \begin{cases}
        u \geq 0 \text{ in } Q_1,\\
        \mathcal{M}^-\left( \diag{|D_iu|^{p/2}} D^2u \diag{|D_iu|^{p/2}} \right)  -\Lambda|Du|^{p+1}\leq f(x) \text{ in } Q_1, \\
        \norm{f}_{L^n(Q_1)} \leq \mu\\
        |\{u>M_1\} \cap Q_1| \geq (1-\delta)|Q_1|,
    \end{cases}
    \end{align*}
    then $u>1$ in $Q_{1/4n}$.
\end{lem}
\begin{proof}
    By considering the inf convolution of $u$,
    \begin{align*}
        u_\epsilon(x)=\inf_{y\in \overline{Q_1}}\left(u(y)+\frac{|x-y|^2}{\epsilon}\right),
    \end{align*}
    we may assume that $u$ is semiconcave and twice differentiable almost everywhere (See \cite{Caffarelli95, Imbert16}). 
    We argue by a contradiction assuming there exists a point $x_0 \in Q_{1/4n}$ such that $u(x_0) \leq 1$.

    We define the vertex set $V = \{u>M \} \cap Q_{1/4n}$ for some $M>1$ to be chosen later, and slide the paraboloid with vertex in $V$ from below until it first touches the graph of $u$.
    Our choice of the `paraboloid' is 
    \begin{align*}
        \varphi(x) = -K\frac{1+p}{2+p} \left(|x_1|^{1+\frac{1}{1+p}} + \cdots +|x_n|^{1+\frac{1}{1+p}} \right)
    \end{align*}
    for some $K>1$ to be determined later.
    Note that $\varphi$ is $C^1$, but not $C^2$ in $n$ hyperplanes $\{z_i=0\}$ for $i\in\{1,\cdots,n\}$.
    Explicitly, we define the touching point set as
    \begin{align*}
        T=\{x\in\overline{Q_1}: \exists y \in V \text{ such that } u(x) - \varphi(x-y) = \min_{z\in \overline{Q_1}} u(z)-\varphi(z-y)  \}.
    \end{align*}
    Note that for $y \in V \subset Q_{1/4n}$ and $z \in \partial Q_1$, we have $u(z)-\varphi(z-y) \geq K\frac{1+p}{2+p}\left(1-\frac{1}{4n}\right)^{2}$ since $u \geq 0$.
    On the other hand, we obtain $u(x_0)-\varphi(x_0-y) \leq 1+K\frac{1+p}{2+p}n\left(\frac{1}{2n}\right)<K\frac{1+p}{2+p}\left(1-\frac{1}{4n}\right)^{2} =:M_1$ by choosing a large $K>1$.
    Hence, the minimum cannot be attained on the boundary $\partial Q_1$ and thus $T \subset Q_1$.
    Moreover, it follows that $T \subset \{  u <M_1\} \cap Q_1$.

    By the differentiability of $u$, for each $x\in T$ there is a unique vertex point $y\in V$ such that $\varphi(z-y)$ touches $u$ from below at $x$.
    We therefore define the mapping $m:T \rightarrow V$ which associates each touching point $x$ with its corresponding vertex $y$.
    Moreover, we get
    \begin{align} \label{grad}
        D_iu(x)= D_i\varphi(x-y) = -K|x_i-y_i|^{\frac{1}{1+p}} \frac{x_i-y_i}{|x_i-y_i|}, \\
        D^2u(x) \geq D^2\varphi(x-y) = -\frac{K}{1+p}\diag{|x_i-y_i|^{-\frac{p}{p+1}}}. \label{hess}
    \end{align}
    Observe that $D^2\varphi(x-y)$ blows up when $x-y \in \bigcup\{z_i=0\}$.
    By \eqref{grad} we can write the map $m$ explicitly as
    \begin{align} \label{vt}
        y_i =x_i + \frac{1}{K^{1+p}}|D_iu|^{p}D_iu(x). 
    \end{align}
    We want to prove that $Du$ is Lipschitz in $T$ in order to apply the area formula to the map $m:x\mapsto y$ as in \cite{Imbert16}.
    However, at critical points where  $|x_i-y_i|=0$ for some $i$, the Hessian $D^2\varphi(x-y)$ blows up, making it difficult to prove the Lipschitz continuity and the inequality \eqref{hess} also cannot be used. 
    In order to deal with this difficulty, we decompose the touching point set according to whether each coordinate direction is degenerate or non-degenerate. 
    For $I \subset \{1,\cdots,n\}$ and $\epsilon>0$, we define
    \begin{align}  \label{defT}
    \begin{split}
        T^{\epsilon}_{I} &= \{ x \in T : |x_i-y_i| >\epsilon \ \text{ for }\ i\in I, \ \ |x_j-y_j|=0 \ \text{ for }\ j\notin I\},\\
        T_{I} &= \{ x \in T : |x_i-y_i| >0 \ \text{ for }\ i\in I, \ \ |x_j-y_j|=0 \ \text{ for }\ j\notin I\},
    \end{split}
    \end{align}
    where $y=m(x)$ and define $V^{\epsilon}_{I} = m(T^{\epsilon}_{I})$ and $V_{I}=m(T_I)$.
    Then for $0<\epsilon_1 < \epsilon_2$, we have
    \begin{align} \label{TV}
    \begin{split}
       T^{\epsilon_2}_I \subset T^{\epsilon_1}_I, \quad T_I = \bigcup_{\epsilon>0} T^\epsilon_I, \quad T=\bigsqcup_{I \subset \{1,\cdots,n\}} T_I,\\
       V^{\epsilon_2}_I \subset V^{\epsilon_1}_I, \quad V_I = \bigcup_{\epsilon>0} V^\epsilon_I, \quad V=\bigcup_{I \subset \{1,\cdots,n\}} V_I.
    \end{split}
    \end{align}
    Our goal is to prove $|V_I| \leq C(|T_I|+\mu^n)$ for each $I \subset \{1,\cdots,n\}$ by applying the area formula.

    \textbf{ Case 1 : $I = \{1,\cdots,n\}$.}
    For simplicity, we denote $T_{\{1,\cdots,n\}} = T_n$ and $V_{\{1,\cdots,n\}} = V_n$.
    Let $\epsilon>0$ be a small number.
    For $x \in T_n^\epsilon$ and $y=m(x)$, there exists a paraboloid $\varphi(z-y)+C_0$ that touches $u$ from below at $z=x$.
    Moreover, since $|x_i-y_i| >\epsilon$ for every $i \in \{1,\cdots,n\}$, $\varphi(z-y)$ is smooth at $z=x$ with $|D^2\varphi(x-y)| \leq C\epsilon^{-2}$.
    Therefore, for any $z \in Q_1$, we get
    \begin{align}
        u(z) &\geq \varphi(z-y) +C_0 \nonumber \\
         &\geq \varphi(x-y) + C_0+ \langle D\varphi(x-y),z-x\rangle - C\epsilon^{-2}|z-x|^2 \nonumber \\
         &=u(x) + \langle Du(x),z-x\rangle - C\epsilon^{-2}|z-x|^2, \label{touchbelow}
    \end{align}
    by the Taylor expansion of $\varphi(\cdot-y)$ at $x$, the fact that $u(x)=\varphi(x-y)+C_0$, and \eqref{grad}.

    Let $x_1, x_2 \in T_n^\epsilon $ and set $y_1 =m(x_1)$ and $ y_2=m(x_2)$.
    For any $z \in B_r(x_1)$ with $r=2|x_1-x_2|$, we have
    \begin{align*}
    u(z) & \geq u(x_1) + \langle Du(x_1),z-x_1 \rangle - C\epsilon^{-2}|z-x_1|^2 \\
    & \geq u(x_2) + \langle Du(x_2),x_1-x_2 \rangle - C\epsilon^{-2}|x_1-x_2|^2 + \langle Du(x_1),z-x_1 \rangle - C\epsilon^{-2}|z-x_1|^2,
    \end{align*}
    by using $\eqref{touchbelow}$ twice.
    By the semiconcavity of $u$, there exists $\tilde{C}>0$ such that
    \begin{align*}
        u(z) \leq u(x_2) + \langle Du(x_2),z-x_2 \rangle +\tilde{C}|z-x_2|^2.
    \end{align*}
    Therefore, using the two inequalities above, we get
    \begin{align*}
        \langle Du(x_1)-Du(x_2),z-x_1 \rangle \leq (C\epsilon^{-2}+\tilde{C})r^2.
    \end{align*}
    Choosing $z$ such that $z-x_1$ is parallel with $Du(x_1)-Du(x_2)$, we obtain
    \begin{align*}
        |Du(x_1)-Du(x_2)| \leq (C\epsilon^{-2}+\tilde{C})|x_1-x_2|.
    \end{align*}
    Thus, we conclude that $Du$ and $m$ are Lipschitz in $T^\epsilon_n$.
    By the area formula on $m:T^\epsilon_n \rightarrow V^\epsilon_n$, we have
    \begin{align} \label{area1}
        |V_n^\epsilon| = \int_{T^\epsilon_n} |\det D_xy| \ dx.
    \end{align}
    We now differentiate \eqref{vt} with respect to $x$ to get
    \begin{align*}
        D_xy = I_n + \frac{1+p}{K^{1+p}}\diag{|D_iu|^{p}}D^2u.
    \end{align*}
    By using $\det (I_n+MN)=\det(I_n+NM)$ for $n\times n$ matrices $M$ and $N$, we have
    \begin{align} \label{det}
        \det D_xy &= \det \left(I_n + \frac{1+p}{K^{1+p}}\diag{|D_iu|^{p}}D^2u \right) \\
        &= \det \left(I_n + \frac{1+p}{K^{1+p}}\diag{|D_iu|^{p/2}}D^2u\diag{|D_iu|^{p/2}}\right) :=\det(A).
    \end{align}
    Note that by direct calculation,
    \begin{align*}
        -\frac{1+p}{K^{1+p}}\diag{|D_i\varphi|^{p/2}}D^2\varphi \diag{|D_i\varphi|^{p/2}}(x-y) =I_n.
    \end{align*}
    Recalling that $Du(x)=D\varphi(x-y)$ and $D^2u(x) \geq D^2\varphi(x-y)$, it follows from equation above that
    \begin{align*}
        A &=I_n + \frac{1+p}{K^{1+p}}\diag{|D_iu|^{p/2}}D^2u \diag{|D_iu|^{p/2}} \\
        &= \frac{1+p}{K^{1+p}}\diag{|D_iu|^{p/2}} (D^2u(x)-D^2\varphi(x-y)) \diag{|D_iu|^{p/2}} \geq 0,
    \end{align*}
    thus $A$ is a nonnegative definite matrix.
    Note that since $|x_i-y_i|>\epsilon$ for any $i \in \{1,\cdots,n\}$, $D^2\varphi(x-y)$ is well defined.
    Therefore, we obtain
    \begin{align*}
        \lambda \tr{A} &= \mathcal{M}^{-}(A) \\
        & \leq \mathcal{M}^{+}(I_n) + \frac{1+p}{K^{1+p}}\mathcal{M}^{-}\left(\diag{|D_iu|^{p/2}}D^2u \diag{|D_iu|^{p/2}}\right) \\
        & \leq n\Lambda +\frac{1+p}{K^{1+p}}(f(x) +|Du|^{p+1}) \\
        & \leq C(1 +f(x)),
    \end{align*}
    where we have used the fact that $|Du(x)|=|D\varphi(x-y)| \leq CK$.
    Thus, we get
    \begin{align*}
        \det D_xy=\det(A) &\leq \left(\frac{\tr{A} }{n}\right)^n \leq C(1 +|f(x)|)^n.
    \end{align*}
    Using \eqref{area1}, we have
    \begin{align*}
        |V_n^\epsilon| = \int_{T^\epsilon_n} |\det D_xy| \ dx \leq \int_{T^\epsilon_n} C(1 +f(x))^n \ dx \leq C(|T^\epsilon_n|+\mu^n).
    \end{align*}
    By letting $\epsilon \rightarrow0$, we have $|V_n| = \lim_{\epsilon\rightarrow0}|V_n^\epsilon| \leq \lim_{\epsilon\rightarrow0}C(|T^\epsilon_n|+\mu^n) \leq C(|T_n|+\mu^n)$.

    \textbf{ Case 2 : $I = \emptyset$.} Note that if $x \in T_\emptyset$, then the corresponding vertex $y \in V_\emptyset$ becomes $y=x$.
    Thus we have $V_\emptyset = T_\emptyset$.
    However, since $V \subset \{u>M\}$ and $T \subset\{u<M\}$, $V_\emptyset = T_\emptyset =\emptyset$.

    \textbf{ Case 3 : $\emptyset\subsetneq I \subsetneq \{1,\cdots,n\}$.} 
    Let $k=|I|$ and $J = \{1,\cdots,n\} \setminus I$.
    Note that by \eqref{vt}, we have $D_ju(x)=0$ for $j \in J$ and $x \in T^\epsilon_I$, so that the equation degenerates in $j$-th direction.
    Let $\epsilon>0$ be small, and let $(a_j)_{j\in J} \in Q^{n-k}_1$ be $(n-k)$ numbers satisfying $a_j \in (-1,1)$.
    We define the $k$-dimensional slice 
    $H(a_j) = \bigcap_{j \in J}\{z_j=a_j\}$, and restrict the function $u, \varphi$ and $m$ to $H(a_j) \cap Q_1^n $ and denote them by $u_{(a_j)}, \varphi_{(a_j)}$ and $m_{(a_j)}$, respectively.
    For notational convenience, for $z \in H(a_j)\cap Q_1^n$, we regard $z$ as an element of $Q_1^k$ by simply ignoring the fixed coordinates $a_j$ with $j \in J$.
    In this way, the restricted functions can be considered as functions defined on the $k$-dimensional cube $Q_1^{k}$.
    We also define $T^{\epsilon}_{I}(a_j) = H(a_j) \cap T^\epsilon_I $ and $V^{\epsilon}_{I}(a_j) =H(a_j) \cap V^\epsilon_I$.
    Then if $x \in T^\epsilon_I(a_j)$, then $y=m(x)$ satisfies $x_j=y_j=a_j$ for any $j \in J$ so that $y \in V^{\epsilon}_{I}(a_j)$.
    Therefore, the restricted map $m_{(a_j)} : T^{\epsilon}_{I}(a_j) \rightarrow V^{\epsilon}_{I}(a_j)$ is well defined with
    \begin{align} \label{vt2}
        y_i =x_i + \frac{1}{K^{1+p}}|D_iu_{(a_j)}|^{p}D_iu_{(a_j)}(x).
    \end{align}
    We next claim that $Du_{(a_j)}$ is Lipschitz in $T^{\epsilon}_{I}(a_j)$.
    Let $x \in T^{\epsilon}_{I}(a_j)$, $y=m(x) \in V^{\epsilon}_{I}(a_j)$ and  $\varphi(z-y)+C_0$ be a paraboloid touching $u$ from below at $z=x$.
    Then the restricted paraboloid $\varphi_{(a_j)}(z-y)+C_0$ also touches $u_{(a_j)}$ below at $z=x$, and so
    \begin{align*}
        Du_{(a_j)}(x)= D\varphi_{(a_j)}(x-y) \in \mathbb{R}^{k}, \quad D^2u_{(a_j)}(x) \geq D^2\varphi_{(a_j)}(x-y)\in S(k).
    \end{align*}
    Recall that, in order to prove Lipschitz continuity of $Du$ in $T^\epsilon$ as in \textbf{Case 1}, we only require semiconcavity of $u$ and $C^2$ regularity of $\varphi(z-y)$ at $z=x$.
    Since $u$ is semiconcave, its restriction $u_{(a_j)}$ is also semiconcave.
    Moreover, the restriction of $\varphi$ satisfies
    \begin{align*}
        \varphi_{(a_j)}(z-y) = -K\frac{p+1}{p+2}\sum_{i\in I} |z_i-y_i|^{1+\frac{1}{1+p}},
    \end{align*}
    because, for $z \in H(a_j)$, we have $z_j=y_j=a_j$ for all $j \in J$.
    Since $|x_i-y_i| >\epsilon$ for $i\in I$, the restricted paraboloid $\varphi_{(a_j)}(z-y)$ is smooth at $z=x$ with $|D^2\varphi_{(a_j)}(x-y)| \leq C\epsilon^{-2}$.
    Therefore, we have
    \begin{align*}
        u_{(a_j)}(z) &\geq \varphi_{(a_j)}(z-y) +C \nonumber \\
         &\geq \varphi_{(a_j)}(x-y) + C+ \langle D\varphi_{(a_j)}(x-y),z-x\rangle - C\epsilon^{-2}|z-x|^2 \nonumber \\
         &=u_{(a_j)}(x) + \langle Du_{(a_j)}(x),z-x\rangle - C\epsilon^{-2}|z-x|^2, 
    \end{align*}
    for any $z \in Q^{k}_1$.
    By the same argument as in \textbf{Case 1}, we conclude that $Du_{(a_j)}$ and $m_{(a_j)}$ are Lipschitz in $T^{\epsilon}_{I}(a_j)$.
    Applying the area formula to the mapping $m_{(a_j)}:T^\epsilon_I(a_j) \rightarrow V^\epsilon_I(a_j)$, we obtain
    \begin{align} \label{area2}
        |V_I^\epsilon(a_j)| = \int_{T^\epsilon_I(a_j)} |\det D_xy| \ dx.
    \end{align}
    We differentiate \eqref{vt2} with respect to $x$ to get
    \begin{align*}
        D_xy = I_{k} + \frac{1+p}{K^{1+p}}\diag{|D_iu_{(a_j)}|^{p}}D^2u_{(a_j)}.
    \end{align*}
    Note that we have
    \begin{align*}
        -\frac{1+p}{K^{1+p}}\diag{|D_i\varphi_{(a_j)}|^{p/2}}D^2\varphi_{(a_j)} \diag{|D_i\varphi_{(a_j)}|^{p/2}}(x-y) =I_{k}.
    \end{align*}
    By the same argument as in \textbf{Case 1},
    \begin{align*}
        A &=I_{k} + \frac{1+p}{K^{1+p}}\diag{|D_iu_{(a_j)}|^{p/2}}D^2u _{(a_j)}\diag{|D_iu_{(a_j)}|^{p/2}} \\
        &= \frac{1+p}{K^{1+p}}\diag{|D_iu_{(a_j)}|^{p/2}} (D^2u_{(a_j)}(x)-D^2\varphi_{(a_j)}(x-y)) \diag{|D_iu_{(a_j)}|^{p/2}} \geq 0,
    \end{align*}
    with $\det(D_xy)=\det(A)$.
    Recall that $D^2\varphi_{(a_j)}(x-y)$ is well defined.
    Therefore, we have
    \begin{align*}
        \lambda \tr{A} &= \mathcal{M}^{-}(A) \\
        & \leq \mathcal{M}^{+}(I_{k}) + \frac{1+p}{K^{1+p}}\mathcal{M}^{-}\left(\diag{|D_iu_{(a_j)}|^{p/2}}D^2u_{(a_j)}\diag{|D_iu_{(a_j)}|^{p/2}} \right).
    \end{align*}
    Observe that for any $M \in S(k)$,
    \begin{align*}
        \sigma  \begin{pmatrix}
             M & 0 \\
             0 & 0_{n-k}
         \end{pmatrix} = \sigma(M) \cup \{\underbrace{0,\cdots,0}_{n-k}\},
    \end{align*}
    where $0_{n-k} \in S(n-k)$ denotes the zero matrix, and the spectrum $\sigma(M)$ is the multiset of eigenvalues of $M$.
    Moreover, simultaneous permutations of rows and columns do not change the spectrum.
    Therefore, since $|D_ju(x)|=0$ for any $j \in J$ and $|Du_{(a_j)}(x)| = |D\varphi_{(a_j)}(x-y)| \leq C $, we have
    \begin{align} \label{IPDE}
    \mathcal{M}^{-}\left(\diag{|D_iu_{(a_j)}|^{p/2}}D^2u_{(a_j)}\diag{|D_iu_{(a_j)}|^{p/2}}\right) -\Lambda|Du_{(a_j)}|^{p+1}  \nonumber \\= \mathcal{M}^{-}\left(\diag{|D_iu|^{p/2}}D^2u \diag{|D_iu|^{p/2}}\right)-\Lambda|Du|^{p+1} \leq f(x),
    \end{align}
    for $x \in T_I^\epsilon$. Hence, we obtain
    \begin{align*}
        \det(D_xy) \leq \left(\frac{\tr{A}}{k}\right)^{k} \leq C(1+|f_{(a_j)}(x)|)^{k},
    \end{align*}
    where $f_{(a_j)}$ is the restriction of $f$ on $H(a_j) \cap Q_1^n$.
    Therefore, by Fubini's theorem and \eqref{area2}, we get
    \begin{align*}
        |V^\epsilon_I|&= \int_{Q_1^{n-k}}|V_I^\epsilon(a_j)| \  da_J  \\
        &= \int_{Q_1^{n-k}}\int_{T^\epsilon_I(a_j)} |\det D_xy| \ dx_I da_J \\
        & \leq \int_{Q_1^{n-k}}\int_{T^\epsilon_I(a_j)} C(1+f_{(a_j)}(x))^{k} \ dx_I da_J \\
        & = \int_{T^\epsilon_I} C(1+f(x))^{k} \ dx \\
        &\leq C(|T^\epsilon_I|+\norm{f}_{L^{k}(Q^n_1)}^{k}) + \leq C(|T^\epsilon_I|+\mu^n).
    \end{align*}
    By letting $\epsilon \rightarrow0$, we have $|V_I| = \lim_{\epsilon\rightarrow0}|V_I^\epsilon| \leq \lim_{\epsilon\rightarrow0}C(|T^\epsilon_I|+\mu^n) \leq C(|T_I|+\mu^n)$.

    \textbf{Conclusion.}
    Therefore, we have $|V_I| \leq C(|T_I|+\mu^n)$ for every $I \subset \{1,\cdots,n\}$.
    Using \eqref{TV}, we have
    \begin{align*}
        |V| \leq \sum_{I \subset \{1,\cdots,n\}} |V^\epsilon_I| \leq C\sum_{I \subset \{1,\cdots,n\}} (|T_I|+\mu^n) = C(|T|+2^n\mu^n).
    \end{align*}
    Moreover, since $V = \{u>M_1\} \cap Q_{1/4n} $ and $T \subset \{ u <M_1\}$, we have
    \begin{align*}
         (1-4^n\delta)|Q_{1/4n}|\leq |V| \leq C(|T|+\mu^n) \leq C(\delta|Q_1|+ \mu^n).
    \end{align*}
    This yields a contradiction if we choose $\delta, \mu>0$ sufficiently small, depending only on $n , p,\lambda, \Lambda$.
    This completes the proof.
    
\end{proof}

\section{barrier function and doubling property} \label{sec5}
In this section, we prove the doubling property
by using a barrier function.
For brevity, for $x =(x_i) \in \mathbb{R}^n$ and $b >0$, we introduce the notation
\begin{align*}
    |x|_b := \sum_{i=1}^n|x_i|^b \in \mathbb{R}, \quad  x ^b:=(|x_i|^{b-1}x_i)  \in \mathbb{R}^n. 
\end{align*}
This convention will be used throughout this section.
Note that the $1/b$-power of $|x|_b$ is the standard $b$-norm.
Observe that $D(|x|_b) = bx^{b-1}$ and $D(x^b) = b\diag{|x_i|^{b-1}}$.
The barrier function is defined by
\begin{align*}
    \Phi_0(x)= \frac{1}{ab}|x|_b^{-a}  = \frac{1}{ab}\left(|x_1|^b +\cdots+|x_n|^b \right)^{-a},
\end{align*}
where $a>1$ is large enough and $b=1+\frac{1}{1+p}$.
The choice of $b$ is related to the `paraboloid' $\varphi(x)$ introduced in Section \ref{sec3}.
Observe that $\Phi_0(z)$ belongs to $C^1( \mathbb{R}^n \setminus \{0\})$ but is not $C^2$ on the set $\bigcup\{z_i=0\}$.
By direct calculation, for $x \in Q_{6n} \setminus Q_{1/8n}$, we have
\begin{align*}
    D\Phi_0(x)&= -|x|_b^{-(a+1)} x^{b-1},\\
    D^2\Phi_0(x) &= b(a+1)|x|_b^{-(a+2)}(x^{b-1} \otimes x^{b-1}) -(b-1)|x|_b^{-(a+1)} \diag{|x_i|^{b-2}}.
\end{align*}
Observe that $D^2\Phi_0(x)$ blows up when $x \in \bigcup\{z_i=0\}$ since $b-2 \leq 0$.
Thus, we get
\begin{align*}
         \diag{|D_i\Phi_0|^{p/2}}D^2\Phi_0\diag{|D_i\Phi_0|^{p/2}}(x) &= |x|_{b}^{-(a+1)(p+1)} \left( b(a+1)\frac{x^{(b-1)(1+p/2)}}{|x|_b^{1/2}} \otimes  \frac{x^{(b-1)(1+p/2)}}{|x|_b^{1/2}}\right. \\
         & \ \ - \left. (b-1) \diag{|x_i|^{b-2 + (b-1)p}}\right) \\
         &=|x|_{b}^{-(a+1)(p+1)} \left( b(a+1)\frac{x^{b/2}}{|x|_b^{1/2}} \otimes  \frac{x^{b/2}}{|x|_b^{1/2}} - (b-1) I_n\right).
\end{align*}
Therefore, since $ \left| \frac{x^{b/2}}{|x|_b^{1/2}}\right|_2 \geq c(b,n)$ and $|D\Phi_0| \leq C(b,n)|x|_b^{-(a+1)}$ for some $c(b,n),C(b,n)>0$, we have
\begin{align} 
    \mathcal{M}^-&\left(\diag{|D_i\Phi_0|^{p/2}}D^2\Phi_0\diag{|D_i\Phi_0|^{p/2}}\right) -\Lambda|D\Phi_0|^{p+1} \nonumber\\&\geq |x|_{b}^{-(a+1)(p+1)}  \left(b(a+1)\mathcal{M}^-\left( \frac{x^{b/2}}{|x|_b^{1/2}} \otimes  \frac{x^{b/2}}{|x|_b^{1/2}} \right) +(b-1) \mathcal{M}^-\left( -I_n\right) -C(b,n)\Lambda\right)  \nonumber \\
    &\geq|x|_{b}^{-(a+1)(p+1)}  \left( (a+1)c(b,n)\lambda - C(b,n)\Lambda \right) > 1 \label{MPhi},
\end{align}
by choosing $a=a(n,p,\lambda, \Lambda)>1$ large enough.

We now prove the following doubling property lemma.
\begin{lem} \label{doublem}
    There exist $\mu<1$ and $M_2>1$ depending on $n,\lambda,\Lambda$ and $p \geq 0$ such that if $u \in C(Q_{6n})$ satisfies
    \begin{align*}
    \begin{cases}
        u \geq 0 \text{ in } Q_{6n},\\
        \mathcal{M}^-\left( \diag{|D_iu|^{p/2}} D^2u \diag{|D_iu|^{p/2}} \right) -\Lambda|Du|^{p+1} \leq f(x) \text{ in } Q_{6n}, \\
        \norm{f}_{L^n(Q_{6n})} \leq \mu,\\
        u>M_2 \text{ in } Q_{1/4n},
    \end{cases}
    \end{align*}
    then $u>1$ in $Q_3$.
\end{lem}
\begin{proof}
    We see that $u$ is semiconcave by using the inf convolution of $u$.
    We argue by contradiction, assuming that $u(x_0) \leq 1$ for some $x_0 \in Q_3$.
    We define the barrier function
    \begin{align} \label{barrier}
        \Phi(x) = K(\Phi_0(x) -\Phi_0(5ne_1)) \quad \text{ in } \mathbb{R}^n \setminus Q_{1/8n},
    \end{align}
    where $K>1$ is chosen large enough so that $\Phi(x) >2$ in $Q_{4}$,
    and extend $\Phi$ smoothly to all of $\mathbb{R}^n$.
    The vertex set is $V = Q_{1/8n}$  and we slide the barrier function with a vertex in $V$ from below.
    Let $T$ denote the set of corresponding touching points.
    Observe that, for $y \in Q_{1/8n}$, we have $\partial Q_{6n} \subset \{z:\Phi(z-y) \leq 0\}$ and $\Phi(x_0-y) > 2$.
    Thus, for $ z \in \partial Q_{6n} $, we have $u(z) - \Phi(z-y) \geq 0$ and $u(x_0) - \Phi(x_0-y) \leq -1$, which implies that $T \subset Q_{6n}$.
    Moreover, by choosing $M_2 = \norm{\Phi}_{L^\infty}$, we have $u(z)-\Phi(z-y) \geq 0$ for $z \in Q_{1/4n}$, which implies $T \subset Q_{6n} \setminus Q_{1/4n}$.

    Let $x \in T$ and $y \in V$ be a corresponding vertex point.
    Since $x \notin Q_{1/4n}$ and $y \in Q_{1/8n}$, the barrier function $\Phi(z-y)$ near $x$ is defined as \eqref{barrier}, and does not involve the extended part.
    By the differentiability of $u$, the vertex point $y\in V$ is uniquely defined and we say $m : T\rightarrow V$ is the function mapping $x$ into $y$.
    Furthermore, we get
     \begin{align} \label{grad2}
        Du(x)&= D\Phi(x-y) = -K|x-y|^{-(a+1)}_b (x-y)^{b-1}, \\
        D^2u(x) &\geq D^2\Phi(x-y).
        \label{hess2}
    \end{align}
    Note that $|Du(x)|>c$ for some $c=c(n,a,p)>0$ since $ |x-y| \leq 7n^2$.
    Thus, using \eqref{grad2} we have
    \begin{align} \label{vt3}
        y&=x - (D\Phi)^{-1}(Du(x)),
    \end{align}
    where $(D\Phi)^{-1}$ is the inverse function of $D\Phi$.
    
    As in the proof of Lemma \ref{measlem}, for $I \subset \{1,\cdots,n\}$ and $\epsilon>0$, we define $T^{\epsilon}_{I}$, $T_{I}$ as in \eqref{defT}
    and $V^{\epsilon}_{I} = m(T^{\epsilon}_{I})$, $V_{I}=m(T_I)$.
    Our goal is to prove $|V_I| \leq C \mu^n$ for each $I \subset \{1,\dots,n\}$.

    \textbf{ Case 1 : $I = \{1,\cdots,n\}$.}
    We write $T_{\{1,\cdots,n\}} = T_n$ and $V_{\{1,\cdots,n\}} = V_n$.
    For $x \in T^\epsilon_n$ and $y =m(x)$, there exists a barrier function $\Phi(z-y)+C$ that touches $u$ from below at $z=x$.
    Since $\Phi(z-y)$ is smooth at $z=x$ and $u$ is semiconcave, by the same argument as in Lemma \ref{measlem}, we conclude that $Du$ and $m$ are Lipschitz continuous in $T^\epsilon_n$.
    
    By differentiating \eqref{vt3} with respect to $x$, 
    \begin{align*}
        D_xy&= I -D((D\Phi)^{-1})(Du(x)) \cdot D^2u(x).
    \end{align*}
    Note that $D((D\Phi)^{-1}) \circ D\Phi(z) = (D^2 \Phi)^{-1}(z)$ by the inverse function theorem.
    By \eqref{grad2}, we get
    \begin{align*}
        D((D\Phi)^{-1})(Du(x)) = D((D\Phi)^{-1}) \circ D\Phi(x-y) = (D^2 \Phi)^{-1}(x-y),
    \end{align*}
    and so
    \begin{align*}
        D_xy&= I- (D^2 \Phi)^{-1}(x-y)  \cdot D^2u(x) \\
        &= (D^2\Phi)^{-1}(x-y) \cdot (D^2\Phi(x-y) -D^2u(x)) \\
        &= D((D\Phi)^{-1})(Du(x)) \cdot (D^2\Phi(x-y) -D^2u(x)).
    \end{align*}
    Recall that $D^2 \Phi(x-y)$ is well defined since $x \in T^\epsilon_n$.
    By direct calculation, we have
    \begin{align*}
        (D\Phi)^{-1}(x)&=  -K^{\frac{1}{ab+1}}|x|^{-\frac{a+1}{ab+1}}_{\frac{b}{b-1}}x^{\frac{1}{b-1}}
        = -K^{\frac{1}{ab+1}}|x|^{-\frac{a+1}{ab+1}}_{2+p}x^{1+p},\\
        D((D\Phi)^{-1})(x)&= -K^{\frac{1}{ab+1}} |x|^{-\frac{a+1}{ab+1}}_{2+p} \diag{|x_i|^{p/2}} B(x)\diag{|x_i|^{p/2}},
    \end{align*}
    where
    \begin{align*}
        B(x)=  \left( -\frac{a+1}{ab+1}(2+p)\frac{x^{1+p/2}}{|x|^{1/2}_{2+p}} \otimes \frac{x^{1+p/2}}{|x|^{1/2}_{2+p}}  +(1+p)I_n\right). 
    \end{align*}
    Note that $B(x)$ is a bounded matrix in the sense that it satisfies $|B_{ij}(x)| \leq C(a,p)$ for any $x\in \mathbb{R}^n$, which implies $|\det(B(x))| \leq C$.
    Using $|Du(x)| > c$, we have
    \begin{align*}
        |\det(D_xy)| &= C|Du|^{-n\frac{a+1}{ab+1}}_{2+p}\left|\det\left( \diag{|D_iu|^{p/2}}B(x) \diag{|D_iu|^{p/2}} ( D^2u(x)-D^2\Phi(x-y))\right)\right| \\
        &\leq C \left|\det\left( \diag{|D_iu|^{p/2}} (D^2u(x)-D^2\Phi(x-y)) \diag{|D_iu|^{p/2}} \right) \right|=:C|\det(A)|.
    \end{align*} 
    Note that by \eqref{hess2}, $A = \diag{|D_iu|^{p/2}} (D^2u(x)-D^2\Phi(x-y)) \diag{|D_iu|^{p/2}} $ is a nonnegative definite matrix.
    Furthermore, using \eqref{grad2} and \eqref{MPhi}, we have
    \begin{align*}
        f(x) &\geq \mathcal{M}^-\left( \diag{|D_iu|^{p/2}} D^2u(x) \diag{|D_iu|^{p/2}} \right) -\Lambda|Du|^{p+1}\\
        &\geq \mathcal{M}^-\left( \diag{|D_iu|^{p/2}} \left(D^2u(x) -D^2\Phi(x-y)\right) \diag{|D_iu|^{p/2}} \right) \\
        &+\mathcal{M}^-\left( \diag{|D_i\Phi|^{p/2}} D^2\Phi(x-y) \diag{|D_i\Phi|^{p/2}} \right)-\Lambda|D\Phi|^{p+1} \\
        &\geq  \lambda\tr{A}.
    \end{align*}
    Thus, we get
    \begin{align*}
        |\det(D_xy)|=C|\det(A)| &\leq \left(\frac{\tr{A}}{n}\right)^n \leq C|f(x)|^n.
    \end{align*}
    Applying the area theorem to the mapping $m:T^\epsilon_n \rightarrow V^\epsilon_n$, we have
    \begin{align*}
        |V_n^\epsilon| = \int_{T^\epsilon_n} |\det D_xy| \ dx \leq \int_{T^\epsilon_n} Cf(x)^n \ dx \leq C\mu^n.
    \end{align*}
    By letting $\epsilon \rightarrow0$, we obtain $|V_n| \leq C \mu^n$.

    \textbf{ Case 2 : $I = \emptyset$.} By the same argument as in Lemma \ref{measlem}, $V_\emptyset = T_\emptyset = \emptyset$.

    \textbf{ Case 3 : $ \emptyset\subsetneq I \subsetneq \{1,\cdots,n\} $.} 
    Let $k=|I|$, $J = \{1,\cdots,n\} \setminus I$, and let $(a_j)_{j\in J} \in Q^{n-k}_{1/8n}$ be $|J|$ numbers satisfying $a_j \in (-1/8n,1/8n)$.
    We define the $k$-dimensional slice $H(a_j) = \bigcap_{j \in J}\{z_j=a_j\}$, and restrict the functions $u, \Phi,m$ and the sets $T^\epsilon_I, V^\epsilon_I$ on $H(a_j) \cap Q_{1/8n}^n$, denoting them by the same notations as in Lemma \ref{measlem}.
    Note that if $x \in T^\epsilon_I(a_j)$, then $y=m(x) \in V^\epsilon_I(a_j)$ since $x_j=y_j=a_j$ for $j \in J$.
    Thus, the restricted map $m_{(a_j)} : T^{\epsilon}_{I}(a_j) \rightarrow V^{\epsilon}_{I}(a_j)$ is well defined.
    Moreover, the restricted barrier $\Phi_{(a_j)}(z-y)$ touches $u_{(a_j)}$ from below at $z=x$, by which
    \begin{align*}
        Du_{(a_j)}(x)= D\Phi_{(a_j)}(x-y) \in \mathbb{R}^{k}, \quad D^2u_{(a_j)}(x) \geq D^2\Phi_{(a_j)}(x-y)\in S(k).
    \end{align*}
    Since $\Phi_{(a_j)}(z-y)$ is smooth at $z=x$ and $u$ is semiconcave, it follows by the same argument as in Lemma \ref{measlem} that $Du_{(a_j)}$ and $m_{(a_j)}$ are Lipschitz continuous in $T^\epsilon_I(a_j)$.
    Moreover, we have
    \begin{align*}
        y&=x - (D\Phi_{(a_j)})^{-1}(Du_{(a_j)}(x)).
    \end{align*}
    Differentiating the above equation with respect to $x$ and using the same argument as in \textbf{Case 1},
    \begin{align*}
        D_xy= D((D\Phi_{(a_j)})^{-1})(Du_{(a_j)}(x)) \cdot (D^2\Phi_{(a_j)}(x-y) -D^2u_{(a_j)}(x)).
    \end{align*}
    Recall that $D^2 \Phi_{(a_j)}(x-y)$ is well defined since $x \in T^\epsilon_I(a_j)$.
    Note that we have
    \begin{align*}
        D((D\Phi_{(a_j)})^{-1})(x)&= -K^{\frac{1}{ab+1}} |x|^{-\frac{a+1}{ab+1}}_{2+p} \diag{|x_i|^{p/2}} B(x) \diag{|x_i|^{p/2}},
    \end{align*}
    where
    \begin{align*}
        B(x)=  \left( -\frac{a+1}{ab+1}(2+p)\frac{x^{1+p/2}}{|x|^{1/2}_{2+p}} \otimes \frac{x^{1+p/2}}{|x|^{1/2}_{2+p}}  +(1+p)I_{k}\right),
    \end{align*}
    which satisfies $|\det(B(x))| \leq C$.
    Using $|Du(x)|>c$, we get
    \begin{align*}
        |\det(D_xy)| \leq C \left|\det\left( \diag{|D_iu_{(a_j)}|^{p/2}} (D^2u_{(a_j)}(x)-D^2\Phi_{(a_j)}(x-y)) \diag{|D_iu_{(a_j)}|^{p/2}} \right) \right|=:C|\det(A)|
    \end{align*} 
    with $A=\diag{|D_iu_{(a_j)}|^{p/2}} \left(D^2u_{(a_j)}(x) -D^2\Phi_{(a_j)}(x-y)\right) \diag{|D_iu_{(a_j)}|^{p/2}} \geq 0$.
    Moreover, by the same calculation as in \eqref{MPhi} replacing $n$ by $k$, we obtain
    \begin{align*}
        \mathcal{M}^-\left( \diag{|D_i\Phi_{(a_j)}|^{p/2}} D^2\Phi_{(a_j)}\diag{|D_i\Phi_{(a_j)}|^{p/2}} \right) -\Lambda|D\Phi_{(a_j)}|^{p+1} >1.
    \end{align*}
    Recalling that $|D_ju(x)| = 0$ for $j \in J$, we find by \eqref{IPDE} that
    \begin{align*}
        f_{(a_j)}(x) &\geq \mathcal{M}^-\left( \diag{|D_iu|^{p/2}} D^2u(x) \diag{|D_iu|^{p/2}} \right) -\Lambda|Du|^{p+1} \\
        &= \mathcal{M}^-\left( \diag{|D_iu_{(a_j)}|^{p/2}} D^2u_{(a_j)}(x) \diag{|D_iu_{(a_j)}|^{p/2}} \right) -\Lambda|Du_{(a_j)}|^{p+1}\\
        &\geq \mathcal{M}^-\left( \diag{|D_iu_{(a_j)}|^{p/2}} \left(D^2u_{(a_j)}(x) -D^2\Phi_{(a_j)}(x-y)\right) \diag{|D_iu_{(a_j)}|^{p/2}} \right) \\
        & +\mathcal{M}^-\left( \diag{|D_i\Phi_{(a_j)}|^{p/2}} D^2\Phi_{(a_j)}(x-y) \diag{|D_i\Phi_{(a_j)}|^{p/2}} \right) -\Lambda|D\Phi_{(a_j)}|^{p+1}  \\
        &\geq  \lambda\tr{A}.
    \end{align*}
    Thus, we get
    \begin{align*}
        |\det (D_xy)|=C|\det(A)| &\leq \left(\frac{\tr{A}}{k}\right)^{k} \leq C|f_{(a_j)}(x)|^{k}.
    \end{align*}
    Therefore, by Fubini's theorem and the area formula applied to $m_{(a_j)} : T^{\epsilon}_{I}(a_j) \rightarrow V^{\epsilon}_{I}(a_j)$, we have 
    \begin{align*}
        |V^\epsilon_I|&= \int_{Q_{1/8n}^{n-k}}|V_I^\epsilon(a_j)| \  da_J  \\
        &= \int_{Q_{1/8n}^{n-k}}\int_{T^\epsilon_I(a_j)} |\det D_xy| \ dx_I da_J \\
        & \leq \int_{Q_{1/8n}^{n-k}}\int_{T^\epsilon_I(a_j)} C|f_{(a_j)}(x)|^{k} \ dx_I da_J \\
        & = \int_{T^\epsilon_I} C|f(x)|^{k} \ dx \leq C\mu^n.
    \end{align*}
    Letting $\epsilon \to 0$, we obtain $|V_I| \leq C \mu^n$.

    \textbf{Conclusion.}
    Therefore, we have $|V_I| \leq C\mu^n$ for every $I \subset \{1,\cdots,n\}$.
    Using \eqref{TV}, we have
    \begin{align*}
        |Q_{1/8n}|=|V| \leq \sum_{I \subset \{1,\cdots,n\}} |V^\epsilon_I| \leq  C2^n\mu^n.
    \end{align*}
    This leads to a contradiction if we choose $\mu > 0$ sufficiently small, depending on $n,\lambda,\Lambda,p$.
\end{proof}
By combining the two lemmas above, Lemma \ref{measlem} and Lemma \ref{doublem}, we obtain the following corollary.
\begin{cor} \label{meascor}
    There exist $\delta<1$, $\mu<1$ and $M>1$ such that for any $k\geq1, r\leq1$ and $u \in C(Q_{6nr})$ satisfying
    \begin{align*}
    \begin{cases}
        u \geq 0 \text{ in } Q_{6nr},\\
        \mathcal{M}^-\left( \diag{|D_iu|^{p/2}} D^2u \diag{|D_iu|^{p/2}} \right) -\Lambda|Du|^{p+1}\leq f(x) \text{ in } Q_{6nr}, \\
        \norm{f}_{L^n(Q_{6nr})} \leq \mu,\\
        |\{u>kM\} \cap Q_r| \geq (1-\delta)|Q_r|,
    \end{cases}
    \end{align*}
    then $u>k$ in $Q_{3r}$.
\end{cor}
\begin{proof}
    Write $M=M_1M_2$ where $M_1$ is the constant in Lemma \ref{measlem} and $M_2$ is the constant in Lemma \ref{doublem}.
    Then by Remark \ref{scaling}, $v(x)=\frac{u(rx)}{kM_2} \in C(Q_{6n})$ satisfies all the assumptions of Lemma \ref{measlem}, and so $v>1$ in $Q_{1/4}$, which implies $u_r(x)=\frac{u(rx)}{k}>M_2$.
    Applying Lemma \ref{doublem} to $u_r$, we obtain $u>k$ in $Q_{3r}$.
\end{proof}

Using the above corollary and Calder\'on-Zygmund cube decomposition (Lemma \ref{CZ}), we have the following $L^\epsilon$ estimate (See \cite{Caffarelli95}, Lemma 4.6).
\begin{thm} \label{Lep}
    There exist $\delta<1$, $\epsilon<1$ and $\mu<1$ such that for any $u \in C(Q_{6nr})$ satisfying
    \begin{align*}
    \begin{cases}
        u \geq 0 \text{ in } Q_{6n},\\
        \mathcal{M}^-\left( \diag{|D_iu|^{p/2}} D^2u \diag{|D_iu|^{p/2}} \right) -\Lambda|Du|^{p+1} \leq f(x) \text{ in } Q_{6n}, \\
        \norm{f}_{L^n(Q_{6n})} \leq \mu,\\
        \inf_{Q_3}u \leq 1,
    \end{cases}
    \end{align*}
    then for any $t > 0$,
    \begin{align*}
        |\{u>t\}\cap Q_1| <Ct^{-\epsilon}.
    \end{align*}
\end{thm}
\begin{proof}
    It is enough to prove the following inequality
    \begin{align} \label{Lepsilon}
        |\{u>M^i\}\cap Q_1| < C(1-\delta)^i. 
    \end{align}
    Let $E_i =\{u>M^i\}\cap Q_1$.
    By $E_i \subset E_1$ and Corollary \ref{meascor} with $k=r=1$, we have $|E_i|\leq|E_1| \leq (1-\delta)|Q_1|$.
    Moreover, for any dyadic cube $Q=Q_r(x_0)$ satisfying $|E_i \cap Q| > \delta|Q|$, by applying Corollary \ref{meascor} with $k=M^{i-1}$, we obtain $\tilde{Q} \subset Q_{3r}(x_0) \subset E_{i-1}$ where $\tilde{Q}$ is the precedessor of $Q$.
    Therefore, by Lemma \ref{CZ}, we have $|E_i| \leq (1-\delta)|E_{i-1}|$, which implies \eqref{Lepsilon} and finishes the proof.
\end{proof}
\begin{proof}[Sketch of proof of Theorem \ref{Main2}.] 
Once we obtain the $L^\epsilon$ estimate, Theorem \ref{Lep}, the remaining parts of the proofs of Theorem \ref{Main2} and Corollary \ref{maincor} are standard and can be found in Chapter 4 of \cite{Caffarelli95} or Sections 6 and 7 of \cite{Imbert16}.
\end{proof}
\textbf{Data Availability}
Data sharing not applicable to this article as no
datasets were generated or analyzed during the current study.

\textbf{Conflict of interest}
The authors declared that they have no conflict of interest to
this work.

\bibliographystyle{amsplain}

\end{document}